\documentclass[11pt,a4 paper,oneside, reqno]{amsart}
\usepackage[utf8]{inputenc}
\usepackage[OT1]{fontenc}
\usepackage[english]{babel}
\usepackage{dsfont}
\usepackage{amsfonts}
\usepackage{amsmath}
\usepackage{bbm}
\usepackage{wasysym}
\usepackage{amsthm}
\usepackage{amssymb}
\usepackage{float}
\usepackage{textcomp}
\usepackage{xcolor}
\usepackage{graphicx}
\usepackage{fancybox}
\usepackage{fancyhdr}
\usepackage{mathrsfs}
\usepackage{tikz}
\usetikzlibrary{arrows}
\usetikzlibrary{calc}
\usepackage{centernot}
\usepackage{listings}
\usepackage{faktor}
\usepackage{bm}
\usepackage{setspace}
\usepackage{hyperref}
\usepackage{newlfont}
\usepackage{geometry}
\usepackage{ccaption}
\usepackage{tipa}
\usepackage[autostyle,italian=guillemets]{csquotes}
\usepackage{comment}

\usepackage{guit}
\usepackage{tikz-cd}
\usepackage{enumerate}

\everymath{\displaystyle}

\theoremstyle{definition}
\newtheorem{Def}{Definition}[section]

\theoremstyle{remark}
\newtheorem{obs}[Def]{Remark}
\theoremstyle{plain}
\newtheorem{prop}[Def]{Proposition}

\newtheorem{lema}[Def]{Lemma}

\newtheorem{cor}[Def]{Corollary}

\newtheorem{teo}[Def]{Theorem}

\newcommand{\bo}{\mathbf}
\newcommand{\A}{{\mathcal A}}
\newcommand{\B}{{\mathcal B}}
\newcommand{\C}{{\mathcal C}}
\newcommand{\D}{{\mathcal D}}
\newcommand{\E}{{\mathcal E}}

\newcommand{\K}{{\mathcal K}}
\renewcommand{\L}{{\mathcal L}}

\newcommand{\V}{{\mathcal V}}

\newcommand{\mt}{\mathscr}
\newcommand{\tx}{\textnormal}

\newcommand{\colim}{\operatornamewithlimits{colim}}

\makeatletter
\newcommand{\changeoperator}[1]{%
	\csletcs{#1@saved}{#1@}%
	\csdef{#1@}{\changed@operator{#1}}%
}
\newcommand{\changed@operator}[1]{%
	\mathop{%
		\mathchoice{\textstyle\csuse{#1@saved}}
		{\csuse{#1@saved}}
		{\csuse{#1@saved}}
		{\csuse{#1@saved}}%
	}%
}
\makeatother

\makeatletter
\def\@tocline#1#2#3#4#5#6#7{\relax
	\ifnum #1>\c@tocdepth % then omit
	\else
	\par \addpenalty\@secpenalty\addvspace{#2}%
	\begingroup \hyphenpenalty\@M
	\@ifempty{#4}{%
		\@tempdima\csname r@tocindent\number#1\endcsname\relax
	}{%
		\@tempdima#4\relax
	}%
	\parindent\z@ \leftskip#3\relax \advance\leftskip\@tempdima\relax
	\rightskip\@pnumwidth plus4em \parfillskip-\@pnumwidth
	#5\leavevmode\hskip-\@tempdima
	\ifcase #1
	\or\or \hskip 1em \or \hskip 2em \else \hskip 3em \fi%
	#6\nobreak\relax
	\hfill\hbox to\@pnumwidth{\@tocpagenum{#7}}\par% <---- \dotfill -> \hfill
	\nobreak
	\endgroup
	\fi}
\makeatother

\changeoperator{sum}
\changeoperator{prod}
\changeoperator{coprod}

\title{On continuity of accessible functors}
\author{Giacomo Tendas}

\address{School of Mathematical and Physical Sciences, Macquarie University NSW 2109, Australia}
\email{giacomo.tendas@mq.edu.au}
\date{\today}
\thanks{}

\begin{document}
	
\begin{abstract}
	We prove that for each locally $\alpha$-presentable category $\K$ there exists a regular cardinal $\gamma$ such that any $\alpha$-accessible functor out of $\K$ (into another locally $\alpha$-presentable category) is continuous if and only if it preserves $\gamma$-small limits; as a consequence we obtain a new adjoint functor theorem specific to the $\alpha$-accessible functors out of $\K$. Afterwards we generalize these results to the enriched setting and deduce, among other things, that a small $\V$-category is accessible if and only if it is Cauchy complete.

\end{abstract}	
	
\maketitle

{

	
	\noindent {\bf Competing interests:} The author declares none.
}
\section{Introduction}

The theory of locally presentable categories has been thoroughly studied and developed since it was first introduced in \cite{GU71:libro}. There are many ways in which locally presentable categories can be described (see for instance \cite{AR94:libro}); most directly one can say that, given a regular cardinal $\alpha$, a category $\K$ is locally $\alpha$-presentable if it is cocomplete and freely generated by a small category under $\alpha$-filtered colimits; a category that satisfies only the latter condition is instead called $\alpha$-accessible. It is a standard result that the locally $\alpha$-presentable categories can alternatively be characterized as those $\alpha$-accessible categories that are also {\em complete}. Because of this variety of characterizations, different notions of morphisms between them have been considered in the literature. Here we define a morphism between locally $\alpha$-presentable categories to consist of a functor which is continuous (preserves all small limits) and $\alpha$-accessible (preserves $\alpha$-filtered colimits); these, together with the natural transformations, identifies a 2-category that we call $\bo{Lp}_\alpha$. Alternatively one could define a morphism between locally $\alpha$-presentable categories to be a cocontinuous functor which preserves the $\alpha$-presentable objects; this describes a 2-category biequivalent to $\bo{Lp}_\alpha^{op}$.

The aim of this paper is to give a characterization of the morphisms out of a locally $\alpha$-presentable category $\K\in\bo{Lp}_\alpha$ in terms of those that preserve $\gamma$-small limits, for some determined $\gamma$. More specifically we prove that for any locally $\alpha$-presentable category $\K$ there exists a regular cardinal $\gamma$ such that an $\alpha$-accessible $F\colon \K\to\L$, with $\L$ locally $\alpha$-presentable, is continuous if and only if it preserves all $\gamma$-small limits (Theorem~\ref{gamm-cont}). The choice of $\gamma$ depends entirely on the category $\K_\alpha$ (see Remark~\ref{optimal}).

This result can also be interpreted as a new adjoint functor theorem for $\alpha$-accessible functors out of a locally $\alpha$-presentable category. The ``ur-adjoint functor theorem'' says that if a category $\K$ has all (possibly large) limits and $U\colon\K\to\L$ preserves them, then $U$ has a left adjoint. When $\K$ only has small limits (as usually happens), then one invokes Freyd's general adjoint functor theorem, which requires $U$ to be continuous and to satisfy in addition the solution set condition; this is the case in particular when $\K$ and $\L$ are locally presentable, and $U$ is a continuous and accessible functor. Then our Theorem~\ref{adjoint} says that the condition on continuity can be weakened to $\gamma$-continuity (for some $\gamma$) when we restrict to the locally $\alpha$-presentable case.

In Section~\ref{Set} we introduce the necessary background notions, prove the main result, and then give a few applications including the new adjoint theorem mentioned above. Then, in Section~\ref{V} we prove an enriched version of the main result based on the notion of locally presentable $\V$-category introduced in \cite{Kel82:articolo}. We obtain again an adjoint functor theorem specialized to the $\alpha$-accessible case (Theorem~\ref{V-adjoint}), and moreover we prove that a small $\V$-category is accessible (in the sense of \cite{BQR98}) if and only if it is Cauchy complete (Theorem~\ref{smallacc}).

\subsection*{Acknowledgements}

I would like to thank my supervisor Steve Lack for suggesting I should write up this paper, for his many helpful comments in the making of, and for going through a preliminary version of it. Thanks also to my associate supervisor Richard Garner for pointing out to me the connection with Boolean algebras in~\ref{boolean}. Finally, let me acknowledge with gratitude the support of an International Macquarie University Research Excellence Scholarship.

\section{The $\bo{Set}$ case}\label{Set}

\subsection{Background}

We assume the reader to be familiar with the notions of locally presentable categories and accessible functors between them, standard references for this are \cite{MP89:libro} and \cite{AR94:libro}.

\begin{Def}
	Let $\alpha$ be a regular cardinal; we say that a functor $M\colon\C^{op}\to\bo{Set}$ is $\alpha$-flat if its left Kan extension $\tx{Lan}_YM\colon[\C,\bo{Set}]\to\bo{Set}$ along the Yoneda embedding is $\alpha$-continuous. We say that $M$ is Cauchy if $\tx{Lan}_YM$ is continuous.
\end{Def}

We are calling a functor continuous if it preserves all small limits, and $\alpha$-continuous if it preserves all the $\alpha$-small ones. Then one can say equivalently that $M$ is Cauchy if and only if it is $\alpha$-flat for every $\alpha$ (indeed $\tx{Lan}_YM$ is continuous if and only if it preserves $\alpha$-small limits for each regular cardinal $\alpha$). 

In the result below we consider the category of elements $\tx{El}(M)$ associated to a functor $M\colon \C^{op}\to\bo{Set}$; this comes together with a projection $\pi\colon \tx{El}(M)\to\C$ (note that in the literature there are two dual notions of categories of elements, the projection $\pi$ univocally identifies the one we consider).

The following is a standard characterization of flat functors:

\begin{prop}[\cite{AR94:libro}]\label{flat-char}
	For a functor $M\colon\C^{op}\to\bo{Set}$ the following are equivalent:\begin{enumerate}\setlength\itemsep{0.25em}
		\item $M$ is $\alpha$-flat;
		\item $\tx{Lan}_YM\colon[\C,\bo{Set}]\to\bo{Set}$ preserves all $\alpha$-small colimits of representables;
		\item $\tx{El}(M)$ is $\alpha$-filtered;
		\item $M$ is an $\alpha$-filtered colimit of representables.
	\end{enumerate}
	If $\C$ is $\alpha$-cocomplete, these are further equivalent to:\begin{enumerate}
		\item[(5)] $M$ is $\alpha$-continuous. 
	\end{enumerate}
\end{prop}

This in turn says that $M\colon\C^{op}\to\bo{Set}$ is Cauchy if and only if its category of elements $\tx{El}(M)$ is absolute; meaning that $\tx{El}(M)$-colimits are absolute colimits, and that $M$ itself is an absolute colimit of representables.

The following Lemma appeared in the enriched context as \cite[Lemma~2.7]{LT21:articolo}, we provide a proof (in the ordinary setting) for completeness. 

\begin{lema}\label{flat-restriction}
	Let $J\colon \B\to\C$ and $ M \colon \B^{op}\to\bo{Set}$ be two functors; then:\begin{enumerate}\setlength\itemsep{0.25em}
		\item if $ M $ is $\alpha$-flat then $\tx{Lan}_{J^{op}} M$ is;
		\item if $J$ is fully faithful and $\tx{Lan}_{J^{op}}M $ is $\alpha$-flat then $ M $ is $\alpha$-flat as well.
	\end{enumerate}
	The same holds when replacing $\alpha$-flat by Cauchy.
\end{lema}
\begin{proof}
	\begin{comment}
	Let $\psi:=\tx{Lan}_{J^{op}} M$; then $J$ induces a final functor $H\colon\tx{El}(M)\to \tx{El}(\psi)$, which is fully faithful if $J$ is. Thus, thanks to Proposition~\ref{flat-char}, condition $(1)$ can be restated as: if $\D$ is $\alpha$-filtered and $H\colon\D\to\E$ is final, then $\E$ is $\alpha$-filtered too. 
	\end{comment}
	By definition a functor $ M $ is $\alpha$-flat if its left Kan extension along the Yoneda embedding $\tx{Lan}_YM\colon [\B,\bo{Set}]\to\bo{Set}$, is $\alpha$-continuous. 
	Note that the triangle below commutes;
	\begin{center}
		\begin{tikzpicture}[baseline=(current  bounding  box.south), scale=2]
			
			\node (a) at (0.7,0.7) {$[\B,\bo{Set}]$};
			\node (c) at (0, 0) {$[\C,\bo{Set}]$};
			\node (d) at (1.8, 0) {$\bo{Set}$};
			
			\path[font=\scriptsize]
			
			(c) edge [->] node [above] {$[J,\bo{Set}]\ \ \ \ \ \ \ \ \ \ \ \  $} (a)
			(a) edge [->] node [above] {$\ \ \ \ \ \ \ \ \ \ \tx{Lan}_YM$} (d)
			(c) edge [->] node [below] {$\tx{Lan}_{Y'}(\tx{Lan}_{J^{op}}M)$} (d);
			
		\end{tikzpicture}	
	\end{center} 
	where $Y'$ is the Yoneda embedding for $\C^{op}$ (this can easily be seen in terms of weighted colimits). As a consequence if $ M$ is $\alpha$-flat then so is $\tx{Lan}_{J^{op}}M$ since $[J,\bo{Set}]$ is continuous.
	Conversely, if $J$ is fully faithful and $\tx{Lan}_{J^{op}}M $ is $\alpha$-flat, then 
	\begin{equation*}
		\begin{split}
			\tx{Lan}_{Y}M&\cong \tx{Lan}_{Y}M\circ id_{[\B,\bo{Set}]} \\
			&\cong \tx{Lan}_{Y}M\circ [J,\bo{Set}]\circ \tx{Ran}_{J}\\
			&\cong \tx{Lan}_{Y'}(\tx{Lan}_{J^{op}}M) \circ \tx{Ran}_{J}\\
		\end{split}
	\end{equation*}
	where $id_{[\B,\bo{Set}]}\cong[J,\bo{Set}]\circ \tx{Ran}_{J}$ since $J$ is fully faithful. It follows that $\tx{Lan}_{Y}M$ is $\alpha$-continuous because $\tx{Lan}_{Y'}(\tx{Lan}_{J^{op}}M)$ is (by assumption) and $\tx{Ran}_{J}$ is continuous.
	
	The last assertion is now a consequence of what we have already proven since $M$ is Cauchy if and only if it is $\alpha$-flat for every $\alpha$.
	
\end{proof}

Note that another way to prove the result above would be by using the category of elements and Proposition~\ref{flat-char}. In that case it would be enough to observe that $(1)$ if $\D$ is $\alpha$-filtered and $H\colon\D\to\E$ is final then $\E$ is $\alpha$-filtered too, and $(2)$ if $\E$ is $\alpha$-filtered and $H\colon\D\to\E$ is fully faithful and final, then $\D$ is $\alpha$-filtered (see Example 1.o on page 63 of \cite{AR94:libro}). However we prefer to keep a more formal approach.

\subsection{The result}

We start this section by recalling the notion of $\alpha$-small functor:

\begin{Def}\cite[4.1]{Kel82:articolo}
	A functor $M\colon\C^{op}\to\bo{Set}$ is called $\alpha$-small if $\C$ is an (essentially) $\alpha$-small category and $M$ lands in $\bo{Set}_\alpha$.
\end{Def}

Note that, assuming we are given an $\alpha$-small category $\C$, to say that $M\colon\C^{op}\to\bo{Set}$ is $\alpha$-small is the same as saying that it is $\alpha$-presentable as an object of $[\C^{op},\bo{Set}]$.

The result below first appeared (with a different choice of $\gamma$) in the proof of \cite[Theorem~2.2.2]{MP89:libro}.

\begin{lema}\label{gamma-flat-small}
	Let $\C$ be such that $\C(B,C)$ is $\beta$-small for any $B$ and $C$, and let $\gamma>\beta$. Then any $\gamma$-flat functor $M\colon\C^{op}\to\bo{Set}$ lands in $\bo{Set}_\beta$. If moreover $\C$ has less than $\gamma$ objects, then $M$ is $\gamma$-small.
\end{lema}
\begin{proof}
	Assume that $M$ doesn't land in $\bo{Set}_\beta$; then we can find $C\in\C$ and a family $(x_i\in M(C))_{i<\beta}$ of cardinality $\beta$ where the $x_i$ are all distinct. Since $M$ is $\gamma$-flat, its category of elements $\tx{El}(M)$ is $\gamma$-filtered. Now, the $x_i$ form a $\gamma$-small family of objects of $\tx{El}(M)$; thus there exists $y\in M D$ and morphisms $f_i\colon C\to D$ in $\C$ such that $M(f_i)(y)=x_i$ for any $i<\beta$. But by hypothesis the $f_i$ can't all be be distinct; contradicting the fact that all the $x_i$ actually are.
\end{proof}

\begin{lema}\label{flat+finite=cauchy}
	Every $\alpha$-small and $\alpha$-flat functor is Cauchy.
\end{lema}
\begin{proof}
	Let $M\colon \C^{op}\to\bo{Set}$ be $\alpha$-small and $\alpha$-flat. Consider the free completion $\Phi_\alpha^\dagger\C$ of $\C$ under $\alpha$-small limits, this comes together with the inclusion $J\colon \C\to\Phi_\alpha^\dagger\C$. Since $M$ is $\alpha$-small, $\tx{El}(M)$ is an essentially $\alpha$-small category; thus  we can consider the limit 
	$$ X:= \lim(\tx{El}(M)\stackrel{\pi}{\longrightarrow} \C \stackrel{J}{\longrightarrow} \Phi_\alpha^\dagger\C) $$ 
	in $\Phi_\alpha^\dagger\C$. Next we prove that $\tx{Lan}_JM$ is isomorphic to the representable $\Phi_\alpha^\dagger\C(X,-)$: since $M$ is $\alpha$-flat the Kan extension $\tx{Lan}_JM$ is $\alpha$-flat too (by Lemma~\ref{flat-restriction}) and hence $\alpha$-continuous; therefore it's enough to prove that $\tx{Lan}_JM$ and $\Phi_\alpha^\dagger\C(X,-)$ coincide when restricted to $\C$:
	\begin{align}
			\Phi_\alpha^\dagger\C(X,J-)&\cong \Phi_\alpha^\dagger\C(\lim_x(J\circ \pi x),J-)\nonumber \\
			&\cong \colim_x\Phi_\alpha^\dagger\C(J\circ \pi x,J-)\tag{1} \\
			&\cong \colim_x\C(\pi x,-)\nonumber \\
			&\cong M(-)\nonumber \\
			&\cong \tx{Lan}_JM(J-)\nonumber
	\end{align}
	as required, where (1) holds since $\Phi_\alpha^\dagger\C(-,JC)$ is $\alpha$-cocontinuous for any $C\in\C$ (property of the free completion). It follows that $\tx{Lan}_JM$ is representable, and hence a Cauchy weight (the left Kan extension of a representable functor along Yoneda is an evaluation map, and hence is continuous); by Lemma~\ref{flat-restriction} then $M$ is Cauchy as well.
\end{proof}

\begin{obs}
	This Lemma is true more generally for a sound class $\mt D$ as in \cite{ABLR02:articolo}: if $M$ is such that $\tx{El}(M)$ is at the same time in $\mt D$ and $\mt D$-filtered, then $M$ is Cauchy.
\end{obs}

As a consequence:

\begin{cor}\label{gammaflat-cauchy}
	Let $\C$ be a small category and $\gamma$ be a regular cardinal as in the last part of Lemma~\ref{gamma-flat-small}. Then every $\gamma$-flat functor $M\colon\C^{op}\to\bo{Set}$ is Cauchy.
\end{cor}
\begin{proof}
	This is a direct consequence of Lemma~\ref{gamma-flat-small} and~\ref{flat+finite=cauchy}.
\end{proof}

In short, it's enough to take $\gamma=\beta^+$, where $\beta$ is such that $\C$ is $\beta$-small, but the description given in the Corollary above might provide a smaller cardinal. For example taking $\K=\bo{Set}$ and $\alpha=\aleph_0$, since $\bo{Set}_f$ is $\aleph_1$-small, then we can certainly consider $\gamma=\aleph_2$; however the hypotheses of Lemma~\ref{gamma-flat-small} provide a better $\gamma$, in fact we can actually choose $\gamma=\aleph_1$ (since $\bo{Set}_f$ has countably many objects and the hom-sets are all finite).

We can now prove the main result of this section:

\begin{teo}\label{gamm-cont}
	Let $\K$ be locally $\alpha$-presentable. There exists a regular cardinal $\gamma$ for which every $\alpha$-accessible and $\gamma$-continuous $F\colon \K\to\L$, with $\L$ locally $\alpha$-presentable, is in fact continuous.
\end{teo}
\begin{proof}
	Let $\gamma$ be the one given in Corollary~\ref{gammaflat-cauchy} for $\C=\K_\alpha^{op}$, and denote by $J\colon\K_\alpha\hookrightarrow\K$ the inclusion.
	
	Notice that a functor $F\colon \K\to\L$ as above is continuous if and only if $\L(A,F-)$ is such for any $A\in\L_\alpha$. Since $\L(A,F-)$ is still $\alpha$-accessible and preserves all the limits that $F$ preserves, we can assume without loss of generality $\L=\bo{Set}$. 
	Therefore, we are given an $\alpha$-accessible and $\gamma$-continuous $F\colon \K\to\bo{Set}$, and we need to prove that it is actually continuous. Since $F$ is $\alpha$-accessible, it is the left Kan extension of its restriction $FJ\colon\K_\alpha\to\bo{Set}$; as a consequence the following triangles commute (up to isomorphism),
	
	\begin{center}
		\begin{tikzpicture}[baseline=(current  bounding  box.south), scale=2]

			\node (a) at (0,0.7) {$[\K_{\alpha}^{op},\bo{Set}]$};
			\node (c) at (0, 0) {$\K$};
			\node (d) at (1.2, 0) {$\bo{Set}$};
			\node (e) at (0,-0.7) {$\K_\alpha$};
			
			\path[font=\scriptsize]
			
			(c) edge [right hook->] node [left] {$\K(J,1)$} (a)
			(e) edge [right hook->] node [left] {$J$} (c)
			(a) edge [->] node [above] {$\ \ \ \ \ \ \ \ \ \tx{Lan}_Y(FJ)$} (d)
			(c) edge [->] node [above] {$F\ \ $} (d)
			(e) edge [->] node [below] {$\ \ \ \ \ \ \ \ \ FJ$} (d);
			
		\end{tikzpicture}	
	\end{center} 
	where the vertical composite is the Yoneda embedding $Y\colon\K_{\alpha}\to [\K_{\alpha}^{op},\bo{Set}]$, and $\K(J,1)$ is continuous since it identifies $\K$ with the $\alpha$-continuous functors out of $\K_\alpha^{op}$. Now, $F$ is $\gamma$-continuous by hypothesis; thus $\tx{Lan}_Y(FJ)$ preserves $\gamma$-small limits of representables (these being $\gamma$-small limits in $\K$) and hence $FJ$ is $\gamma$-flat by Proposition~\ref{flat-char}. By our choice of $\gamma$, it follows from Corollary~\ref{gammaflat-cauchy} that $FJ$ is Cauchy and thus $\tx{Lan}_Y(FJ)$ is continuous. Therefore $F$ is continuous being the composite of two continuous functors.
\end{proof}

\begin{obs}\label{optimal}
	The optimal regular cardinal $\gamma$ provided by our proofs is one for which:\begin{itemize}
		\item $\K_\alpha$ has less than $\gamma$ objects (up to isomorphism);
		\item there exists $\beta$ such that $\gamma>\beta>\#\K(X,Y)$ for each $X,Y\in\K_\alpha$.
	\end{itemize}
\end{obs}

By taking $\K=\bo{Set}$, $\alpha=\aleph_0$, and $\gamma=\aleph_1$ (thanks to the comments above the Theorem), it follows that a finitary functor $F\colon\bo{Set}\to\bo{Set}$ is continuous if and only if it preserves countable products and equalizers.

\subsection{Some Applications}

\subsubsection{An(other) adjoint functor theorem}

Freyd's {\em general adjoint functor theorem} says that if $\K$ is complete and $F\colon\K\to\L$ is continuous and satisfies the solution set condition, then it has a left adjoint. In the context of locally presentable categories this implies that every continuous and accessible functor between locally presentable categories has a left adjoint \cite[Theorem~1.66]{AR94:libro}. Our result is a specialization of this to the case of $\alpha$-accessible functors: 

\begin{teo}\label{adjoint}
	Let $\K$ be locally $\alpha$-presentable. There exists a regular cardinal $\gamma$ such that for any $\alpha$-accessible $U\colon \K\to\L$, with $\L$ locally $\alpha$-presentable, the following are equivalent:\begin{enumerate}
		\item $U$ has a left adjoint;
		\item $U$ is $\gamma$-continuous.
	\end{enumerate}
\end{teo}
\noindent Note that $\gamma$ can be chosen again as in Remark~\ref{optimal}.

\subsubsection{Dualizable modules}

Let $\K=R\tx{-}\bo{Mod}$ be the monoidal category of $R$-modules for a commutative ring $R$, and $\alpha=\aleph_0$. Then we can use Theorem~\ref{gamm-cont} to characterize the dualizable $R$-modules (that is, the dualizable objects of $R\tx{-}\bo{Mod}$); these can also be described as the finitely generated projective $R$-modules.

First let's focus on the optimal choice of $\gamma$:\begin{itemize}
	\item if $R$ if finite, then $R\tx{-}\bo{Mod}_f$ has countably many objects and its hom-sets are all finite; so we can choose $\gamma=\aleph_1$;
	\item if $\alpha=\#R$ is infinite, then $R\tx{-}\bo{Mod}$ has countably many objects but its hom-sets have cardinality $\alpha$. Thus we can take $\gamma=\alpha^{++}$. 
\end{itemize}

Let $M$ be an $R$-module; then $M$ is dualizable if and only if the functor $M\otimes-\colon R\tx{-}\bo{Mod}\to R\tx{-}\bo{Mod}$ is continuous (it should actually be continuous as an enriched $R\tx{-}\bo{Mod}$-functor, but these are equivalent conditions, see also Section~\ref{dualizable}). Since every functor $M\otimes-$ is cocontinuous (and hence finitary), a consequence of Theorem~\ref{gamm-cont} is:

\begin{prop}
	An $R$-module $M$ is dualizable if and only if it is flat and $M\otimes-$ preserves $\gamma$-small products.
\end{prop}  

For a finite $R$ this is saying that $M$ needs to be flat and $M\otimes-$ needs to preserve countable products. Note however that the choice of $\gamma$ is not optimal in general: for $R=\mathbb{Z}$ it's enough to take $\gamma=\aleph_1$ (easy to check), while by the results above we are only given $\gamma=\aleph_2$.

This result about dualizable modules seems to be new, as far as the author knows. A result which is similar in style (but different in content) can be found in \cite{Bass63:articolo}. There Bass proves that any {\em projective} $R$-module that satisfies some conditions involving a fixed cardinal is actually {\em free}; instead we prove that any {\em flat} $R$-module that satisfies some (other) conditions involving a fixed cardinal is actually {\em dualizable}.

\subsubsection{Recognising the $\alpha$-presentables}\label{boolean}

Let $\K$ be a locally $\alpha$-presentable category for which $\K_{\alpha}$ is $\alpha$-complete. Let $\widehat{\K}:=\tx{Ind}_\alpha(\K_\alpha^{op})$ be the free cocompletion of $\K_{\alpha}^{op}$ under $\alpha$-filtered colimits, so that $\widehat{\K}\simeq\alpha\tx{-Cont}(\K_{\alpha},\bo{Set})$; denote by $H\colon \K_{\alpha}\to\K$ and $J\colon \K_{\alpha}^{op}\to\widehat{\K}$ the inclusions. Then we can consider the following composite:
\begin{center}
	\begin{tikzpicture}[baseline=(current  bounding  box.south), scale=2]
		
		\node (a) at (0,0) {$\K$};
		\node (c) at (1.6, 0) {$\alpha\tx{-Cont}(\K_\alpha^{op},\bo{Set})$};
		\node (d) at (3.7, 0) {$\alpha\tx{-Filt}(\widehat{\K},\bo{Set})$};
		
		\path[font=\scriptsize]
		
		(a) edge [->] node [above] {$\K(H,1)$} (c)
		(c) edge [->] node [above] {$\tx{Lan}_J$} (d);
		
	\end{tikzpicture}	
\end{center} 
where $\K(H,1)$ is actually an equivalence, and $\alpha\tx{-Filt}(\widehat{\K},\bo{Set})$ is the full subcategory of $[\widehat{\K},\bo{Set}]$ spanned by the functors that preserve $\alpha$-filtered colimits (which is equivalent to $[\K_\alpha^{op},\bo{Set}]$). Note that $\tx{Lan}_J$ is fully faithful and its essential image is given by those $\alpha$-accessible functors which are also $\alpha$-continuous.

Call the composite of these $G\colon \K\to \alpha\tx{-Filt}(\widehat{\K},\bo{Set})$. It's easy to see that if $X\in\K_{\alpha}$ then $GX\cong\widehat{\K}(X,-)$ is representable and hence continuous; vice versa if $F\colon\widehat{\K}\to\bo{Set}$ is continuous and preserves $\alpha$-filtered colimits then $F\cong \widehat{\K}(X,-)$ for some $X\in\widehat{\K}_{\alpha}\simeq\K_{\alpha}^{op}$; thus $F\cong GX$. Let now $\gamma$ be as in Remark~\ref{optimal}, then:

\begin{prop}
	An object $X$ of $\K$ is $\alpha$-presentable if and only if $GX$ preserves $\gamma$-small products.
\end{prop}
\begin{proof}
	Use the results above and Theorem~\ref{gamm-cont}, plus the fact that $GX$ preserves $\gamma$-small limits if and only if it preserves $\gamma$-small products (since it already preserves finite limits).
\end{proof}

As an example, consider $\alpha=\aleph_0$ and $\K=\bo{Bool}$ the category of boolean algebras and morphisms between them. Now, since $\bo{Bool}_f\simeq\bo{Set}_f^{op}$, it follows that $\widehat{\bo{Bool}}=\bo{Set}$; moreover for any $B\in\bo{Bool}$ the functor $GB\colon\bo{Set}\to\bo{Set}$ defined above can be described as the one sending a set $X$ to the set
$$ GB(X)=\{f\colon X\to B\tx{ of finite support }|\ \bigvee_i fi=1\tx{ and }\ fi\wedge fj=0 \ \forall i\neq j \} $$
since this preserves filtered colimits and, with this definition of $G$, we have $GB(X)\cong\bo{Bool}(2^X,B)$ for any finite set $X$. Then by Remark~\ref{optimal} we can choose $\gamma=\aleph_1$, and hence the proposition above says that a boolean algebra $B$ is finite if and only if the functor $GB$ preserves countable products. The endofunctor $GB$ is actually a monad on $\bo{Set}$ whose algebras are the $B$-sets (see \cite{Ber91:articolo} where $GB(X)$ is denoted by $X[B]^*$).

\section{The enriched case}\label{V}

We now fix $\V=(\V_0,\otimes,I)$ to be a symmetric monoidal closed and locally presentable category, and consider $\alpha_0$ such that $\V$ is locally presentable as a closed category in the sense of \cite{Kel82:articolo} (such $\alpha_0$ exists by \cite[Proposition~2.4]{KL2001:articolo}). From now on every cardinal will be assumed to be greater or equal to $\alpha_0$.

We follow the terminology of \cite{Kel82:libro} for general enriched concepts, and we assume the reader to be familiar with the notion of locally presentable $\V$-category as well as that of $\alpha$-small weighted limit, introduced in \cite{Kel82:articolo}. Accordingly, we say that a $\V$-functor is $\alpha$-continuous if it preserves all the $\alpha$-small weighted limits.

Then we can immediately generalize Theorem~\ref{gamm-cont} to the enriched setting:

\begin{teo}\label{V-gamma-cont}
	Let $\K$ be a locally $\alpha$-presentable $\V$-category. There exists a regular cardinal $\gamma$ for which every $\alpha$-accessible and $\gamma$-continuous $F\colon \K\to\L$, with $\L$ locally $\alpha$-presentable, is in fact continuous.
\end{teo}
\begin{proof}
	It's enough to consider $\gamma$ as in Theorem~\ref{gamm-cont} for $\K_0$, which is locally $\alpha$-presentable as an ordinary category. 
	
	Indeed, if $F$ is $\gamma$-continuous in the weighted sense, then $F_0\colon\K_0\to\L_0$ is $\gamma$ continuous as an ordinary functor, and hence continuous by Theorem~\ref{gamm-cont}. It follows then that $F$ preserves all conical limits and powers by $\gamma$-presentable objects. This is enough to ensure that $F$ is continuous since each object of $\V$ is a conical colimit of $\gamma$-presentable ones, thus each power in $\K$ is a conical limit of powers by $\gamma$-presentables, which are preserved by $F$.   
\end{proof}

The notion of $\alpha$-flat and $\alpha$-small $\V$-functor has been considered in the enriched setting again in \cite{Kel82:articolo}, where the ordinary $\alpha$-small limits were replaced by the $\alpha$-small weighted ones. So that $M\colon\C^{op}\to\V$ is called $\alpha$-flat if $\tx{Lan}_YM\colon[\C,\V]\to\V$ is $\alpha$-continuous; while it is called $\alpha$-small if $\C$ has less than $\alpha$ objects (up to isomorphism), $\C(C,D)\in\V_\alpha$ for any $C,D\in\C$, and $M$ lands in $\V_\alpha$.

Cauchy $\V$-functors have also been widely used and, like in the ordinary setting, can be characterized as those $\V$-functors whose left Kan extension along the Yoneda embedding is continuous, as well as those that are weights for absolute colimits (see for example \cite{KS05:articolo,Str1983:articolo}).

We are now ready to deduce the enriched analogue of Corollary~\ref{gammaflat-cauchy}.

\begin{cor}\label{V-gammaflat-cauchy}
	Let $\C$ be a small $\V$-category; then there exists $\gamma$ such that every $\gamma$-flat $\V$-functor $M\colon\C^{op}\to\V$ is Cauchy.
\end{cor}
\begin{proof}
	The weight $M$ is $\gamma$-flat if and only if $\tx{Lan}_YM\colon[\C,\V]\to\V$ is $\gamma$-continuous, and is Cauchy if and only if $\tx{Lan}_Y M$ is continuous. Thus it's enough to take $\gamma$ as in Theorem~\ref{V-gamma-cont} for $\K=[\C,\V]$, $\alpha=\alpha_0$, and $F=\tx{Lan}_YM$ (which is cocontinuous, and hence $\alpha$-accessible).
\end{proof}

Note also that Lemma~\ref{flat+finite=cauchy} has an enriched version:

\begin{lema}\label{V-flat+finite=cauchy}
	Every $\alpha$-small and $\alpha$-flat weight is Cauchy.
\end{lema}
\begin{proof}
	Let $M\colon \C\to\V$ be $\alpha$-small and $\alpha$-flat. Consider the free completion $\Phi_\alpha^\dagger\C$ of $\C$ under $\alpha$-small weighted limits, this comes together with the inclusion $J\colon \C\to\Phi_\alpha^\dagger\C$. Since $M$ is $\alpha$-small we can consider the limit $X\colon =\{M,J\}$ in $\Phi_\alpha^\dagger\C$. Then we prove that $\tx{Lan}_JM$ is isomorphic to the representable $\Phi_\alpha^\dagger\C(X,-)$: since $M$ is $\alpha$-flat the Kan extension $\tx{Lan}_JM$ is $\alpha$-flat too and hence $\alpha$-continuous; therefore it's enough to prove that $\tx{Lan}_JM$ and $\Phi_\alpha^\dagger\C(X,-)$ coincide when restricted to $\C$:
	\begin{equation*}
		\begin{split}
			\Phi_\alpha^\dagger\C(X,J-)&\cong \Phi_\alpha^\dagger\C(\{M,J\},J-)\\
			&\cong M\square*\Phi_\alpha^\dagger\C(J\square,J-)\\
			&\cong M\square*\C(\square,-)\\
			&\cong M(-).
		\end{split}
	\end{equation*}
	Since also $\tx{Lan}_JM$ restricts to $M$ we are done. It follows that $\tx{Lan}_JM$ is a Cauchy weight (since every representable functor is); by \cite[Lemma~2.7]{LT21:articolo} then $M$ is Cauchy too.
\end{proof}

\subsection{Some Applications}

\subsubsection{An(other) adjoint functor theorem}

As in the ordinary case we obtain an adjoint functor theorem specialized to the $\alpha$-accessible $\V$-functors. This is again a consequence of Theorem~\ref{V-gamma-cont} and of the fact that every continuous and accessible $\V$-functor between locally presentable $\V$-categories has a left adjoint \cite[Theorem~8.6]{BQR98}.

\begin{teo}\label{V-adjoint}
	Let $\K$ be a locally $\alpha$-presentable $\V$-category. There exists a regular cardinal $\gamma$ such that for any $\alpha$-accessible $\V$-functor $U\colon \K\to\L$, with $\L$ locally $\alpha$-presentable, the following are equivalent:\begin{enumerate}
		\item $U$ has a left adjoint;
		\item $U$ is $\gamma$-continuous.
	\end{enumerate}
\end{teo}
Note that $\gamma$ can be chosen again as in Remark~\ref{optimal} for $\K_0$.

\subsubsection{Dualizable objects}\label{dualizable}

Recall that an object $X\in\V$ is called dualizable if there exist $X^*\in\V$ and morphisms $\eta_X\colon I\to X\otimes X^*$ and $\epsilon_X\colon X^*\otimes X\to I$, called unit and counit respectively, satisfying the triangle equalities. Equivalently, since $\V$  is closed, $X$ is dualizable if and only if there exists $X^*\in\V$ such that $X\otimes - \cong [X^*,-]\colon \V_0\to\V_0$. By \cite[Section~6]{KS05:articolo}, this is equivalent to $X\otimes-$ being continuous. Then a direct application of Theorem~\ref{V-gamma-cont} to $F=M\otimes-$ gives:

\begin{prop}
	There exists a regular cardinal $\gamma$ such that an object $X\in\V$ is dualizable if and only if $X\otimes-\colon \V\to\V$ is $\gamma$-continuous.
\end{prop}

The following is an application of Lemma~\ref{V-flat+finite=cauchy}:

\begin{prop}
	Let $\V$ be locally $\alpha$-presentable as a closed category. An object $X\in\V$ is dualizable if and only if:\begin{enumerate}
		\item $X$ is $\alpha$-presentable;
		\item $X$ is $\alpha$-flat, or equivalently $X\otimes -$ is $\alpha$-continuous.
	\end{enumerate}
\end{prop}

\subsubsection{Small accessible $\V$-categories}

As an application of the main Theorem we can prove a generalization of \cite[Proposition~2.6]{AR94:libro} which shows that a small ordinary category is accessible if and only if it has splittings of idempotents, or, equivalently, if it is Cauchy complete.

In the enriched context there are two main notions of accessibility; here we consider the one introduced in \cite{BQ96:articolo,BQR98}, where a $\V$-category is called accessible if it is freely generated by a small $\V$-category under $\alpha$-flat weighted colimits, for some $\alpha$. Equivalently, $\A$ is accessible if and only if $\A\simeq\alpha\tx{-Flat}(\C^{op},\V)$ for some small $\V$-category $\C$ and a regular cardinal $\alpha$ \cite[Theorem~5.3]{BQR98}, where $\alpha\tx{-Flat}(\C^{op},\V)$ is the full subcategory of $[\C^{op},\V]$ spanned by the $\alpha$-flat $\V$-functors.

On the other hand, recall that a $\V$-category is called Cauchy complete if it has all colimits weighted by Cauchy $\V$-functors; this is equivalent to saying that every Cauchy $M\colon\C^{op}\to\V$ is representable. We are now ready to prove:

\begin{teo}\label{smallacc}
	A small $\V$-category is accessible if and only if it is Cauchy complete.
\end{teo}
\begin{proof}
	Every accessible $\V$-category is Cauchy complete by \cite[Corollary~5.5]{BQR98}. For the opposite direction consider any Cauchy complete and small $\V$-category $\C$. By Corollary~\ref{V-gammaflat-cauchy} we can find $\gamma$ such that every $\gamma$-flat $\V$-functor $M\colon \C^{op}\to\V$ is Cauchy. Since $\C$ is Cauchy complete, this means that every $\gamma$-flat weight out of $\C^{op}$ is representable; therefore $\C\simeq\gamma\tx{-Flat}(\C^{op},\V)$ is accessible.
\end{proof}

\end{document}